\documentclass[11pt,reqno]{amsproc}

\usepackage[margin=1.0in]{geometry}

\usepackage[utf8]{inputenc}
\usepackage{amsmath,amssymb,amsthm, comment, mathtools}
\usepackage{amsthm}
\usepackage{hyperref}
\usepackage{mathrsfs}
\usepackage{slashed,xcolor}
\usepackage{units}

\usepackage{etoolbox}
\expandafter\patchcmd\csname\string\Praoof\endcsname
  {\normalparindent}{0pt }{}{}

\newtheorem*{thm*}{Theorem}

\newcommand{\na}{\nabla}

\newcommand{\R}{{\mathbb R}}
\newtheorem{thm}{Theorem}
\newtheorem{prop}{Proposition}

\newtheorem{lemma}{Lemma}
\newtheorem{cor}{Corollary}

\theoremstyle{definition}
\newtheorem{rem}{Remark}

\newcommand{\ve}{{\varepsilon}}

\newcommand{\rmd}{{\rm d}}

\newcommand{\be}{\begin{equation}}
\newcommand{\ee}{\end{equation}}

\newcommand{\p}{\partial}

\newcommand{\ba}{\begin{array}{l}}
\newcommand{\ea}{\end{array}}

\newcommand{\Pra}{{\mathsf {Pr}}}
\newcommand{\LL}{{\mathsf{L}_{\mathsf {s}}}}
\newcommand{\cs}{{c_{\mathsf {s}}}}
\newcommand{\Ra}{{\mathsf {Ra}}}
\newcommand{\Nu}{{\mathsf {Nu}}}

\title[]{Bounds on heat flux for Rayleigh-B\'enard convection\\ between Navier-slip fixed-temperature boundaries}

\author{Theodore D. Drivas}
\address{Department of Mathematics, Stony Brook University,
Stony Brook, NY, 11794}
\email{tdrivas@math.stonybrook.edu}
\author{Huy Q. Nguyen}
\address{Department of Mathematics, Brown University, Providence, RI 02912}
\email{huy$\_$quang$\_$nguyen@brown.edu}

\author{Camilla Nobili}
\address{Departement Mathematik, Universit\"at Hamburg, Germany }
\email{ camilla.nobili@uni-hamburg.de }

\date{today}

\date{today}
\begin{document}

\begin{abstract}
We study two-dimensional  Rayleigh-B\'enard convection with Navier-slip, fixed temperature boundary conditions and establish bounds on the Nusselt number. As the slip-length varies with Rayleigh number $\Ra$, this estimate interpolates between the Whitehead--Doering bound by $\Ra^{\frac{5}{12}}$ for free-slip conditions \cite{WD11} and classical Doering--Constantin  $\Ra^{\frac{1}{2}}$ bound \cite{DC96}.
\end{abstract}

\maketitle

\section{Introduction}

The standard Rayleigh-B\'enard convection model describes the dynamics of a fluid layer confined between two rigid plates held at different uniform temperatures: the lower plate is hot and the upper plate is cool. 
This temperature difference triggers density variations of the fluid layers and instability ensues, leading to a convective fluid motion and, as the   control parameter Rayleigh number $\Ra$ increases,  eventually becomes turbulent. 
Rayleigh-B\'enard convection is a paradigm of nonlinear dynamics, including pattern formation and fully developed turbulence and has important applications in meteorology, oceanography and industry. 
A principal quantity of interest due to its relevance in geophysical and industrial applications is the vertical heat transport across the domain. This is usually expressed through the non-dimensional  Nusselt number  $\Nu$, which is  the ratio between the total heat flux and the flux due to thermal conduction. Famously, experiment and numerical simulation suggest a power-law scaling for the Nusselt number $\Nu$
$$\Nu\sim \Pra^{\alpha}\Ra^{\beta}\, \quad \mbox{ for some } \alpha, \beta\in \R\,,$$
where $\Ra$ and $\Pra$ are the non-dimensional Rayleigh and Prandtl number, respectively. In \cite{GL2000} a systematic theory for the scaling of the Nusselt number $\Nu$ is proposed, based on the decomposition of the global thermal and kinetic energy dissipation rates into their
boundary layer and bulk contributions.
As such, it is of  interest to provide mathematical constraints on allowed exponents from the equations of motion.

In physical theories, scaling laws are based, in part, on the structure of (thermal and viscous) boundary layers. It is therefore interesting to understand how the heat transport properties change with respect to different choice of boundary conditions for the velocity.
Most research has focused on the cases where the velocity field satisfies the no-slip  \cite{CN016, CD99,DC96,W08} and free-slip boundary conditions \cite{Wa20,W20,WD11,WD12}.
In this paper we consider the nondimensional Rayleigh-B\'enard convection model subject to Navier-slip boundary conditions. We note that, in contrast to the free-slip boundary conditions studied by Whitehead--Doering, the Navier-slip boundary conditions allow for vorticity to be produced at the boundary.  In a sense, these conditions interpolate between the no-slip and free-slip conditions as the slip length is increased from $0$ to $\infty$.  As such our bounds degenerate to those available for no-slip in the small slip length regime.
As we show later in this paper, the bound
$\Nu\lesssim \Ra^{\frac 12}$,
holds uniformly in Prandtl number in any dimension and for any boundary conditions such that the vertical component of the velocity is zero at the (upper and lower) boundaries.  At fixed $\Pr$, this bound corresponds to the classical Spiegel--Kraichnan scaling and has since been termed the ``ultimate regime''. To this day, there is active debate regarding the validity of the ultimate regime insofar as it can be inferred from data \cite{JD09,D19,Z18}.
We remark that the bound holds in any dimension and for any of the three type of boundary conditions mentioned above and its estimation uses only non-penetration of the velocity at the walls.

We now describe our setup precisely. Let $\Omega = [0,\Gamma]\times[0,1]$ be the channel with boundaries at $\{x_2=0\}$ and $\{x_2=1\}$ and periodic in $x_1$. We consider the Rayleigh-B\'enard system \cite{OPN17}
\begin{align} \label{e1}
\frac{1}{\Pra}(\partial_t u+u\cdot\nabla u)+\nabla p-\Delta u&=\Ra Te_2, \qquad \text{in}\ \  \Omega\,,\\  \label{e2}
\nabla\cdot u&=0, \qquad \quad\quad \  \text{in}\ \  \Omega,\\  \label{e3}
\partial_tT+u\cdot \nabla T&=\Delta T, \qquad\quad\ \text{in}\ \  \Omega,\\  \label{e4}
\partial_2 u_1&=\frac{1}{\LL} u_1, \qquad \ \  \text{on}\ \  \{x_2=0\},\\  \label{e5}
-\partial_2 u_1&=\frac{1}{\LL} u_1, \qquad \ \  \text{on}\ \   \{x_2=1\},\\  \label{e6}
u_2&=0, \qquad \ \quad \ \ \  \text{on}\ \  \{x_2=0\}\cup \{x_2=1\},\\  \label{e7}
T&=1, \qquad\quad \ \ \ \  \text{on}\ \  \{x_2=0\},\\  \label{e8}
T&=0, \qquad\quad \ \ \ \  \text{on}\ \  \{x_2=1\}.
\end{align}
In the horizontal direction $x_1$, all the unknowns are $\Gamma$-periodic.  See Figure \ref{fig:RB} for a depiction of the setup in 2d.  For higher dimensions, $e_2$ in equation \eqref{e1} becomes $e_d$ and the boundary conditions are  \eqref{e5}--\eqref{e6} in all tangential components. There are two nondimensional parameters appearing in the system: the Rayleigh number $\Ra$ which expresses the strength of the thermal forcing and the Prandtl number $\Pra$ which represents the ratio of kinematic viscosity to thermal diffusivity. 

\begin{figure}[h!]
   \centering
\includegraphics[width=0.67\textwidth]{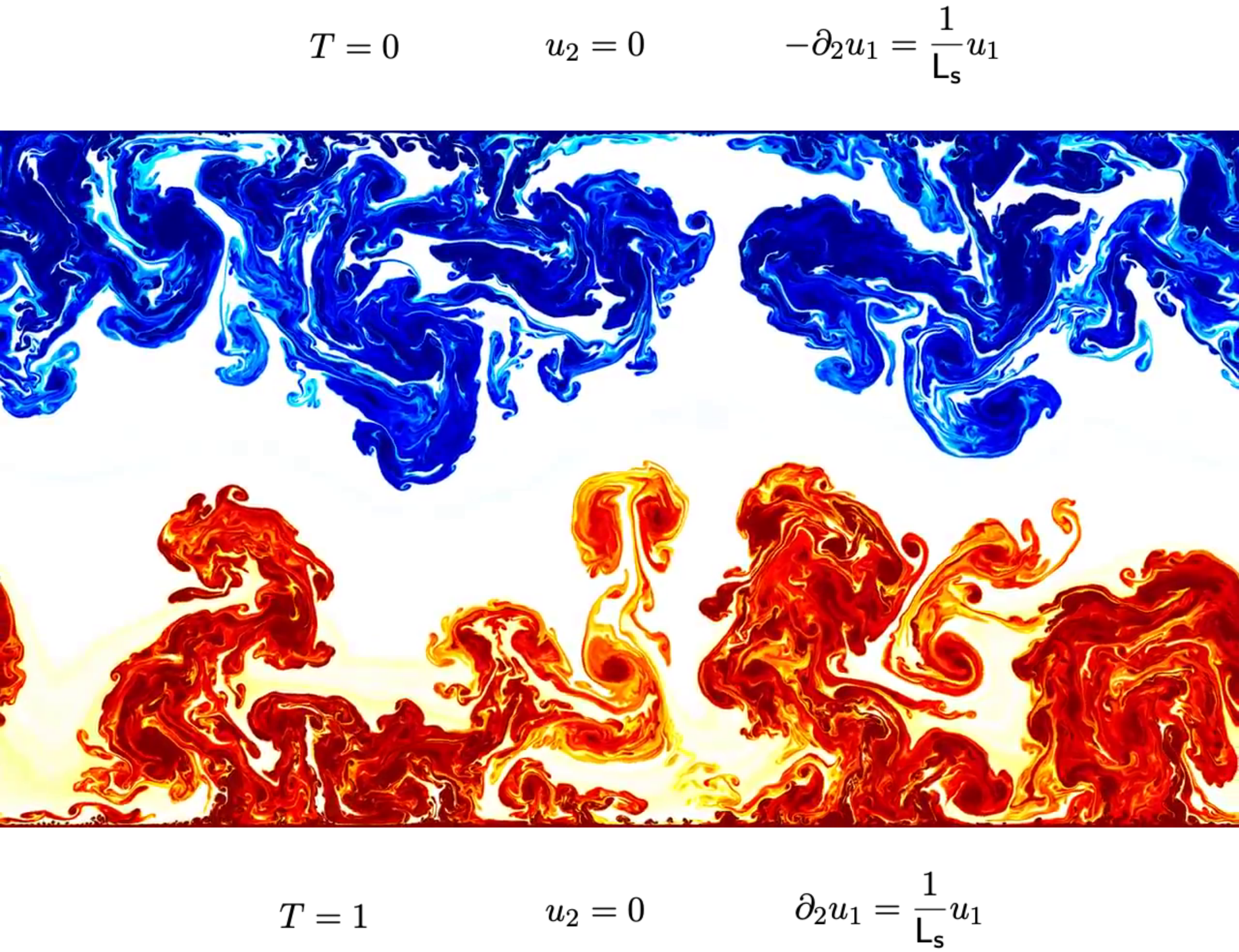}
    \caption{Visualization (with data from no-slip convection \cite{L15}) of temperature field.}   \label{fig:RB}
\end{figure}

As \eqref{e1}--\eqref{e8} is already non-dimensional, the Nusselt number is defined simply by
\begin{equation}\label{Nu-def}
\Nu :=\langle u_2 T-\partial_2 T\rangle, 
\end{equation}
 where we have introduced notation for the long-time, global-in-space average
\be
\langle \varphi \rangle=\limsup_{T \rightarrow \infty}\frac{1}{T }\int_0^{T}\frac{1}{\Gamma}\int_{0}^{\Gamma}\int_0^1\varphi (x_1,x_2,t) \, \rmd x_2\, \rmd x_1\, \rmd t.
\ee
We shall also write $\langle \varphi \rangle_{x_j}$ for the long-time and $x_j$ average. Our main result is the following:
\begin{thm}\label{theorem}\label{maintheorem}
Let $\LL>0$. Then
\begin{itemize}
\item For any $d\geq 2$, we have
\be
 \Nu \lesssim \Ra^{\frac{1}{2}}.
\ee
\item  For $d=2$, if $\Pr$ satisfies $\LL ^2\Pra^2\geq \Ra^{\frac{3}{2}}$, then  for all $\Ra>1$ it holds
 \be\label{theo:mainbound}
 \Nu \lesssim \Ra^{\frac{5}{12}}+\LL^{-2}\Ra^{\frac 12}.
 \ee
 \end{itemize}
 The implicit constants depend only on $\Gamma$, $\| T_0\|_{L^\infty}$ and $\|u_0\|_{W^{1, r}}$ for any fixed $r\in (2, \infty)$.
  \end{thm}
Note that when $\LL =  \cs\Ra^\alpha$ with $\cs>0$ then for $ \Pra\geq\cs^{-1}\Ra^{\frac{3}{4}-\alpha}$ the bound \eqref{theo:mainbound} reads
 \be\label{Nubd}
 \Nu \lesssim \Ra^{p(\alpha)}, \qquad p(\alpha):=\begin{cases} \frac{5}{12} &\text{if } \alpha \geq  \frac{1}{24}\\
\frac{1}{2}-2\alpha &\text{if } 0 \leq  \alpha \leq \frac{1}{24}
 \end{cases}.
 \ee
Theorem \ref{theorem} recovers the Whitehead--Doering bound of \cite{WD11} in 2d with $\LL=\infty$ and of \cite{WD12} in 3d with $\LL=\Pra=\infty$.  For smaller slip-lengths, the bound \eqref{Nubd} approaches the classical result of Doering--Constantin \cite{DC96}.  Our result improves upon available bounds at fixed Prandtl numbers when the system is equipped with no-slip boundary conditions instead of \eqref{e4}--\eqref{e5} provided that the slip-length is sufficiently large $\LL\geq \cs \Ra^{\frac{3}{4}}$, suggesting that the Navier-slip conditions may slightly inhibit turbulent heat transport. We remark that the work of Choffrut-Nobili-Otto \cite{CN016} for no-slip boundaries (in arbitrary dimensions) gives $\Nu\lesssim \Ra^{\frac 13}$ for $\Pr\gtrsim \Ra^{\frac 13}$, which improves the bound over Doering--Constantin in that regime.  Similar arguments may improve our estimates in that case. 
Moreover we observe that
for the 3d model with free-slip boundary conditions, Wang and Whitehead proved the estimate $\Nu\lesssim \Ra^{\frac{5}{12}}+\rm{Gr^2}\Ra^{\frac 14}$ where the Grashof number $\rm{Gr}=\frac{\Ra}{\Pr}$ is small.
 
\begin{rem}[Infinite Prandtl number]
For  $d\geq 2$, $\Pra=\infty$, J. Whitehead (unpublished) proved $ \Nu \lesssim \Ra^{\frac{5}{12}}$ for all $\LL>0$.  In Remark \ref{whiteheadproof}, we show how this follows from our argument.
\end{rem}

Inspired by \cite{WD11}, we employ the background field method with the simple ansatz of a background profile $\tau(x_2)$ being constant in the bulk and linear in the boundary layers of size $\delta$. 
Since the Navier-slip conditions allow vorticity production at the walls, our argument is delicate in a number of places compared to that for free-slip conditions. 
A consequence of the vorticity production at the walls is the lack of conservation of the mean of  $u_1$. As a result, our uniform-in-time bound for the kinetic energy grows linearly with the slip-length $\LL$ (see Lemma \ref{bndu} and Remark \ref{rema:Poincare}). Another consequence is that the uniform-in-time bound for the enstrophy does not follow directly from an energy estimate for the vorticity equation. Here, following an idea in \cite{LNP05}, we establish the  uniform $L^p$ bounds
\be\label{intro:Lp}
\|\omega(t)\|_{L^p}  \le C\Big(\|\omega_0\|_{L^p} +\frac{1}{\LL}\| u_0\|_{L^2} +\Ra\Big) \qquad \forall t>0,~p\in [1, \infty).
\ee
 Firstly, \eqref{intro:Lp} yields the long-time average enstrophy balance \eqref{aventrophy}. Secondly, \eqref{intro:Lp} is carefully combined with an appropriate pressure estimate (see \eqref{bound:p}) to handle the bad boundary term in \eqref{Q:9} in such a way that our Nusselt bound \eqref{theo:mainbound} recovers the result in \cite{WD11} when $\LL\to \infty$. 

Following \cite{WD11}, we use the long-time average energy/enstrophy balances and reduce the proof of \eqref{theo:mainbound} to establishing the positivity of certain quadratic functional $\mathcal{Q}$ (see Prop. \ref{qdef})  when parameters are suitably chosen. By obtaining a new estimate for the term $\langle \tau' u_2\theta\rangle$ generated by the background field, we bypass a Fourier argument in \cite{WD11} and base the proof entirely in physical space. 
\section{Energy Identities and Uniform Bounds}

In what follows, we always consider smooth initial data so that the system \eqref{e1}-\eqref{e8} has a unique global smooth solution.  See e.g. \cite{CMR98,HWWXY18}. 
We will repeatedly use that 
$\|T(t)\|_{L^\infty(\Omega)}\leq \max\{1,\|T_0\|_{L^\infty(\Omega)}\}$ for all $t\geq 0$ by the maximum principle. Without loss of generality, we consider initial data $\|T_0\|_{L^\infty(\Omega)}\leq 1$ so that 
\begin{equation}\label{max-prin}
\|T\|_{L^\infty(\Omega)}\le 1.
\end{equation}
Now we recall the well-known (see e.g. \cite{DC96}) identification of the Nusselt number with the heating rate
\begin{prop}\label{nussident}
 The Nusselt number satisfies $\Nu=\langle |\nabla T|^2\rangle$.
\end{prop}
\begin{proof}
Multiplying the temperature equation \eqref{e3} by $T$, integrating by part in space, and using the incompressibility condition \eqref{e2} and the boundary conditions for $u_2$ and $T$, we get  
\[
\frac{1}{2} \frac{\rmd}{\rmd t} \|T\|_{L^2(\Omega)}^2 = - \|\nabla T\|_{L^2(\Omega)}^2 - \int_0^\Gamma \partial_2 T \big|_{x_2=0} \rmd x_1.
\]
Since $\| T(t)\|_{L^2(\Omega)}$ is uniformly bounded in $t$,  averaging in time yields
\[
\langle |\nabla T|^2\rangle  =-\langle\partial_2 T\big|_{x_2=0}  \rangle_{x_1},
\]
where $\langle \cdot \rangle_{x_1}$ denotes the long time and $x_1$ average. On the other hand, if we integrate \eqref{e3} in $x_1$ and time average, we find $\partial_2 \langle u_2 T - \partial_2 T\rangle_{x_1} =0$. Integrating in $x_2$ gives 
\[
\langle u_2 T - \partial_2 T\rangle_{x_1} =\langle (u_2 T - \partial_2 T)\big|_{x_2=0}\rangle_{x_1} =\langle -\partial_2 T\big|_{x_2=0}\rangle_{x_1}.
\]
In view of the definition \eqref{Nu-def}, we deduce that $\Nu=-\langle\partial_2 T\big|_{x_2=0}\rangle_{x_1} =\langle |\nabla T|^2\rangle $.
\end{proof}

\begin{prop}[Energy Balance]\label{balance:u}
Strong solutions of \eqref{e1}--\eqref{e8} satisfy the balance
\be\label{energyeq}
\frac{1}{2\Pra} \frac{\rmd}{\rmd t} \|u\|_{L^2}^2 + \|\nabla u \|_{L^2}^2  + \frac{1}{\LL} \left(\|u_1\|_{L^2(\{x_2=1\})}^2 + \|u_1\|_{L^2(\{x_2=0\})}^2\right)=  \Ra \int_\Omega u_2 T\rmd x.
\ee
\end{prop}
\begin{proof}
Dotting equation \eqref{e1} with $u$, integrating over $\Omega$ and using \eqref{e2} and \eqref{e6}, we find
\[
\frac{1}{2\Pra} \frac{\rmd}{\rmd t} \|u\|_{L^2}^2 = \int_\Omega u \cdot \Delta u + \Ra \int_\Omega u_2 T\rmd x.
\]
Using the  periodicity and \eqref{e4}, \eqref{e5} and \eqref{e6} gives
\begin{align*}
 \int_\Omega u\cdot \Delta u\ \rmd x&=-\| \na u\|_{L^2}^2+\int_0^\Gamma \left(u\cdot \p_2 u \Big|_{x_2=1}-u\cdot \p_2 u\Big|_{x_2=0}\right)\rmd x_1\\
 & = -\| \na u\|_{L^2}^2+\int_0^\Gamma\left( \partial_2 u_1 u_1\Big|_{x_2=1}- \partial_2 u_1 u_1\Big|_{x_2=0}\right) \rmd x_1\\
  & = -\| \na u\|_{L^2}^2- \frac{1}{\LL} \left(\|u_1\|_{L^2(\{x_2=1\})}^2 + \|u_1\|_{L^2(\{x_2=0\})}^2\right).
\end{align*}
 \end{proof}
 
From the energy balance, we find that the kinetic energy is bounded for all times.
 \begin{lemma}\label{bndu}
The energy of $u$ satisfies the following bound
 \be\label{uniformbound:u}
 \|u(t)\|_{L^2}\leq \| u_0\|_{L^2}e^{-t\frac 23\Pra\min\{1, \frac{1}{\LL}\}} +3\Gamma\max\{ 1, {\LL}\}  \Ra,\qquad\forall t>0.
 \ee
  \end{lemma}
 \begin{proof}
 From the fundamental theorem of calculus we have
\[
|u_1(x_1, x_2)|^2\le 2|u_1(x_1, 0)|^2+2\left(\int_0^{x_2}|\p_2 u_1(x_1, y)| \rmd y\right)^2\le 2|u_1(x_1, 0)|^2+2x_2 \int_0^1|\p_2 u_1(x_1, y)|^2\rmd y
\]
and thus upon integrating over $\Omega$, we obtain 
\begin{align*}
\| u_1\|_{L^2(\Omega)}^2
&\le 2\| u_1(\cdot, 0)\|_{L^2(0, \Gamma)}^2+2\| \p_2 u_1\|_{L^2(\Omega)}^2.
\end{align*}
Combing this with the  Poincar\'e inequality $\| u_2\|_{L^2(\Omega)}\le \| \p_2 u_2\|_{L^2(\Omega)}$, we obtain 
\[
\| u\|_{L^2(\Omega)}^2\le 2\| u_1(\cdot, 0)\|_{L^2(0, \Gamma)}^2+3\| \nabla u\|_{L^2(\Omega)}^2.
\]
In addition,  the temperature $T$ obeys the maximum principle \eqref{max-prin}; hence   $\|T(t)\|_{L^2}\leq |\Omega|\|T(t)\|_{L^\infty}\le \Gamma$. Then, Proposition  \ref{balance:u} gives
\be
\frac{1}{\Pra} \frac{\rmd}{\rmd t} \|u\|_{L^2}  \leq   2 \Ra\Gamma- \frac 23\min\{ 1, \tfrac{1}{\LL}\} \| u \|_{L^2}.
\ee
 \end{proof}
 \begin{rem}\label{rema:Poincare}
 Consider the free-slip boundary conditions $u_2=0$ and $\p_2u_1=0$ on $x_2=0,\, 1$, which can be formally obtained by setting $\LL= \infty$ in \eqref{e6}-\eqref{e7}. The (spatial) mean of $u_1$ is conserved upon integrating the first component of \eqref{e1}.  Appealing to the Galilean symmetry of the system, one can assume without loss of generality that the mean of $u_1$ is zero for all time.  Consequently, the Poincar\'e inequality $\| u\|_{L^2}\le C\| \nabla u\|_{L^2}$ holds. Then, the energy  balance
 \[
 \frac{1}{2\Pra} \frac{\rmd}{\rmd t} \|u\|_{L^2}^2 = -\| \nabla u\|_{L^2}^2 + \Ra \int_\Omega u_2 T\rmd x
 \]
 yields the uniform bound $\| u(t)\|_{L^2}\le e^{-\frac{t}{C}}\| u_0\|_{L^2}+C\| T_0\|_{L^\infty}\Ra$. This bound is better than \eqref{uniformbound:u} by the factor $\LL$ in front of $\Ra$. On the other hand, for the Navier-slip boundary condition, the mean of $u_1$ is not conserved due to the generation of vorticity at the walls.
  \end{rem}
\begin{cor}[Average Energy Balance] \label{aveenbal}
The following balance holds
\be
\langle |\nabla u|^2\rangle+ \frac{1}{\LL}\left( \langle u_1^2\big|_{x_2=1} \rangle + \langle u_1^2\big|_{x_2=0} \rangle \right)  =  \Ra( \Nu-1).
\ee
\end{cor}
\begin{proof}
Using the boundary conditions for the temperature \eqref{e7}--\eqref{e8}, one finds
\be
\Nu=1+\langle u_2T\rangle
\ee 
from the definition \eqref{Nu-def}. 
Then the claim follows upon integrating \eqref{energyeq} in time and taking the long time limit using the uniform  bound for $\| u(t)\|_{L^2}$  given by Lemma \ref{bndu}.
\end{proof}

\begin{prop}[Pressure-Poisson equation]\label{presseqn}
The pressure in \eqref{e1} satisfies 
\begin{align}  \label{pe1}
\Delta p &= -\frac{1}{\Pra} \nabla u ^T: \nabla u  +   \Ra \partial_2 T \qquad\quad \  \ \text{in}\ \  \Omega\,,\\  \label{pe2}
-\partial_2  p&=\frac{1}{\LL} \partial_1 u_1- \Ra \qquad\qquad\qquad \ \ \ \ \ \ \  \text{on}\ \  \{x_2=0\},\\  \label{pe3}
\partial_2 p&=\frac{1}{\LL} \partial_1 u_1 \qquad\qquad\qquad\qquad \ \ \  \ \ \ \ \   \text{on}\ \   \{x_2=1\}.
\end{align}
\end{prop}
\begin{proof}
The equation \eqref{pe1} follows from taking the divergence of the momentum equation.  The boundary conditions come from tracing the second component of the momentum equation along the boundaries.  Specifically, one has
\[
\partial_2 p  = \partial_2^2 u_2 + \Ra T=- \partial_1 \partial_2 u_1 + \Ra T,
\]
where $\p_2u_1$  is given by \eqref{e4} and \eqref{e5}.
\end{proof}
\begin{prop}\label{proppress}
For any $r\in (2, \infty)$, there exists $C=C(r, \Gamma)$ such that
\be\label{bound:p}
\| p\|_{H^1(\Omega)} \le C\Big(\frac{1}{\LL} \|\partial_1 \omega\|_{L^2(\Omega)}+\Ra \| T\|_{L^2(\Omega)}+ \frac{1}{\Pr}\| \omega\|_{L^2(\Omega)}\|\omega\|_{L^r(\Omega)}\Big).
 \ee
\end{prop}
\begin{proof}
  On one hand, using the boundary conditions \eqref{pe2}--\eqref{pe3} gives
\begin{align*}
\int_\Omega p \Delta p\  \rmd x&= - \|p\|_{\dot{H}^1}^2 + \int_0^\Gamma p \partial_2 p|^{x_2=1}_{x_2=0} \rmd x_1\\
&= - \| \nabla p\|_{L^2}^2 + \frac{1}{\LL} \int_0^\Gamma \left(p \partial_1 u_1 |_{x_2=1} +p \partial_1 u_1|_{x_2=0} \right)\rmd x_1- \Ra  \int_0^\Gamma p |_{x_2=0} \rmd x_1.
\end{align*}
On the other hand, using  \eqref{pe1}, \eqref{e7} and \eqref{e8}, we find
\be\nonumber
\begin{aligned}
\int_\Omega p \Delta p  \ \rmd x
&=- \frac{1}{\Pr}\int_\Omega p \nabla u^T: \nabla u \ \rmd x  -\Ra\int_\Omega \partial_2 p T \rmd x -\Ra\int_0^\Gamma p |_{x_2=0} \rmd x_1.
\end{aligned}
\ee
Consequently,
\begin{align*}
\| \nabla p\|_{L^2}^2&= \frac{1}{\LL} \int_0^\Gamma \left(p \partial_1 u_1 |_{x_2=1} +p \partial_1 u_1|_{x_2=0} \right)\rmd x_1+\frac{1}{\Pr}\int_\Omega p \nabla u^T: \nabla u\ \rmd x   +\Ra\int_\Omega \partial_2 p T \rmd x.
\end{align*}
By virtue of the Sobolev trace inequality and H\"older's inequality, it follows that 
\begin{align*}
 \|\nabla p\|_{L^2}^2&\lesssim  \frac{1}{\LL}  \|p\|_{{H}^1}\|\partial_1 u_1\|_{{H}^1}+\Ra \|p\|_{{H}^1}\| T\|_{L^2}+ \frac{1}{\Pr}  \| p\nabla u^T:\nabla u\|_{L^1}.
\end{align*}
For any $r\in (2, \infty)$, letting $\frac{1}{q}=\frac{1}{2}-\frac{1}{r}$, we have $q\in (2, \infty)$ and 
\[
\| p\nabla u^T:\nabla u\|_{L^1}\le \| p\|_{L^q}\| \nabla u\|_{L^2}\| \nabla u\|_{L^r}\le C\| p\|_{H^1}\| \omega\|_{L^2}\| \omega\|_{L^r},
\]
where we use the Sobolev embedding and \eqref{elliptic:u2:2}.

Since $p$ has mean zero, we have $ \| p\|_{H^1}\le  C\| \nabla p\|_{L^2}$, so that upon using $\p_1u_1=-\p_2u_2$ we get 
\begin{align*}
 \|p\|_{H^1}&\lesssim  \frac{1}{\LL}  \|\partial_2 u_2\|_{{H}^1}+\Ra \| T\|_{L^2}+ \frac{C}{\Pr} \| \omega\|_{L^2}\| \omega\|_{L^r}.
\end{align*}
From  Lemma \ref{vortidents} and \eqref{elliptic:u2:2}, $\| \nabla u\|_{L^2}=\| \omega\|_{L^2}$ and $ \|\partial_2 u_2\|_{{H}^1}\le C\| \partial_1\omega\|_{L^2}$, whence \eqref{bound:p} follows. 
\end{proof}

\begin{prop}[Vorticity formulation]
The vorticity $\omega = \nabla^\perp \cdot u$ where $\nabla^\perp = (-\partial_2,\partial_1)$ satisfies
\begin{align} \label{ve1}
\frac{1}{\Pra}(\partial_t \omega +u\cdot \nabla \omega)-\Delta \omega&=\Ra \partial_1 T \qquad  \text{in}\ \  \Omega\,,\\  \label{ve2}
-\omega&=\frac{1}{\LL} u_1 \qquad \ \  \text{on}\ \  \{x_2=0\},\\  \label{ve3}
\omega&=\frac{1}{\LL} u_1 \qquad \ \  \text{on}\ \   \{x_2=1\}.
\end{align} 
\end{prop}
\begin{proof}
Equation \eqref{ve1} follows from taking the curl of the momentum equation \eqref{e1}.  The boundary conditions \eqref{ve2}--\eqref{ve3} follow from the conditions \eqref{e4}--\eqref{e5} since the vorticity on the boundary is simply $\omega= -\partial_2 u_1$ upon recalling  \eqref{e6}.
\end{proof}

\begin{lemma}\label{normvort}
The  normal derivative of vorticity satisfies
\be  \label{dve1}
-\partial_2\omega=\frac{1}{\Pra}\big( \partial_t u_1+ u_1\partial_1 u_1\big) + \partial_1 p, \quad  \text{on} \quad  \{x_2=0\} \ \text{and} \  \{x_2=1\}.
\ee
\end{lemma}
\begin{proof}
 Using incompressibility of $u$, we find
\begin{align}\label{d2vort}
\partial_2 \omega &=-\partial_2^2 u_1 + \partial_1\partial_2 u_2= -\Delta u_1.
\end{align}
From the first component of \eqref{e1} traced on the boundary (using $u_2=0$ there) we  have
\be\label{d1pressureident}
 \Delta u_1 = \frac{1}{\Pra}\big(\partial_t u_1+ u_1\partial_1 u_1\big) + \partial_1 p.
\ee
\end{proof}

\begin{prop}[Enstrophy Balance]\label{enstophbal}
The following identity holds
\begin{equation}\label{enstrophyeq}
\begin{aligned}
\frac{1}{2 \Pra}\frac{\rmd}{\rmd t} \|\omega\|_{L^2}^2 &+ \frac{1}{2 \LL\Pra}\frac{\rmd}{\rmd t} \left(\|u_1\|_{L^2(\{x_2=1\})}^2 + \|u_1\|_{L^2(\{x_2=0\})}^2\right)+ \|\nabla \omega\|_{L^2}^2 \\
&=  \frac{1}{\LL} \left( \int_0^\Gamma  p\partial_1 u_1 \Big|_{x_2=1} \rmd x_1+ \int_0^\Gamma   p\partial_1 u_1\Big|_{x_2=0} \rmd x_1\right) +  \Ra\int_\Omega \omega \partial_1 T \rmd x.
\end{aligned}
\end{equation}
\end{prop}
\begin{proof}
Multiplying  \eqref{ve1} by $\omega$ and integrating over the domain, we obtain
\be
\frac{1}{2 \Pra}\frac{\rmd}{\rmd t} \|\omega\|_{L^2}^2 =\int_\Omega \omega \Delta \omega \rmd x+  \Ra\int_\Omega \omega \partial_1 T \rmd x,
\ee
where we have use the non-penetration boundary conditions for the velocity \eqref{e6}.  Now note that 
\begin{align*}
\int_\Omega \omega \Delta \omega \rmd x&=  -\|\nabla \omega\|_{L^2}^2 + \int_0^\Gamma \omega \partial_2 \omega \Big|_{x_2=1} \rmd x_1 -  \int_0^\Gamma \omega \partial_2 \omega\Big|_{x_2=0} \rmd x_1\\
&=  -\|\nabla \omega\|_{L^2}^2 +\frac{1}{\LL}  \int_0^\Gamma u_1 \partial_2 \omega \Big|_{x_2=1} \rmd x_1 + \frac{1}{\LL}  \int_0^\Gamma u_1 \partial_2 \omega\Big|_{x_2=0} \rmd x_1\\
&=  -\|\nabla \omega\|_{L^2}^2 -\frac{1}{2\LL\Pra} \frac{\rmd}{\rmd t}\left( \int_0^\Gamma u_1^2  \Big|_{x_2=1} \rmd x_1+ \int_0^\Gamma u_1^2  \Big|_{x_2=0} \rmd x_1\right) \\
&\qquad + \frac{1}{\LL} \left( \int_0^\Gamma \partial_1 u_1 p  \Big|_{x_2=1} \rmd x_1+ \int_0^\Gamma \partial_1 u_1 p   \Big|_{x_2=0} \rmd x_1\right),
\end{align*}
where we have used Lemma \ref{normvort} together with periodicity of the function $u_1$ in $x_1$.  
\end{proof}

Next we provide uniform in time bounds for the vorticity

\begin{lemma}[$L^p$ vorticity bounds]\label{unifenstbdd}
Let  $\LL\geq 1$, $p\in [1, \infty)$. There is $C= C(p, \Gamma)<\infty$ so that
\begin{align}\label{vorticity:ub}
 \|\omega(t)\|_{L^p}  \le C\Big(\|\omega_0\|_{L^p} +\frac{1}{\LL}\| u_0\|_{L^2} +\Ra\Big) \qquad \forall t>0.
\end{align}
\end{lemma}
\begin{proof}
Since $\Omega$ is bounded it suffices to prove \eqref{vorticity:ub} for $p\in (2, \infty)$. To this end, we follow a strategy used in \cite{LNP05}.  For arbitrary $T>0$ set
\[
\Lambda :=\frac{1}{\LL} \| u_1 \|_{L^\infty(\{x_2=0,1\}\times (0,T))}
\]
and consider the problems
\begin{align*}
\frac{1}{\Pra}(\partial_t \tilde\omega_\pm +u\cdot \nabla \tilde{\omega}_\pm)-\Delta \tilde\omega_\pm&=\Ra \partial_1 T \qquad \text{in}\ \  \Omega\,,\\
\tilde\omega_\pm|_{t=0}&=\pm|\omega_0|  \qquad \ \ \text{in}\ \  \Omega, \\
\tilde\omega_\pm&=\pm\Lambda \qquad \ \ \ \  \text{on}\ \  \{x_2=0\}\cup \{x_2=1\}.
\end{align*} 
Now let $\omega'_\pm:= \omega- \tilde{\omega}_\pm$. This quantity satisfies
\begin{align*} 
\frac{1}{\Pra}(\partial_t \omega'_\pm +u\cdot \nabla {\omega}'_\pm)-\Delta \omega'_\pm&= 0\qquad\qquad\ \  \quad \quad \  \text{in}\ \  \Omega\,,\\
\omega'_\pm|_{t=0}&=\omega_0\mp|\omega_0|  \ \ \qquad \ \ \ \text{in}\ \  \Omega \\
-\omega'_\pm&=\frac{1}{\LL} u_1\pm\Lambda \qquad \ \  \ \ \text{on}\ \  \{x_2=0\},\\ 
\omega'_\pm&=\frac{1}{\LL} u_1\mp \Lambda \qquad \ \ \ \  \text{on}\ \   \{x_2=1\}.
\end{align*} 
By the maximum principle, we have  $\omega'_+ \leq 0$ and  $\omega'_-\geq 0$ a.e. $\Omega\times [0,T)$. Thus we obtain $ \tilde\omega_-\leq \omega\leq \tilde\omega_+$  and hence
\be\label{minmaxeqn}
 |\omega|\leq \max \{ |\tilde\omega_+|,|\tilde\omega_-|\} \qquad  \text{a.e.} \qquad \Omega\times [0,T).
\ee
We now bound $\tilde\omega_\pm$ in $L^p$.  We focus on $\tilde\omega= \tilde\omega_+$, the other is similar. Let $\hat{\omega} :=\tilde{\omega} -\Lambda$.  This solves
\begin{align*} 
\frac{1}{\Pra}(\partial_t \hat{\omega}+u\cdot \nabla \hat{\omega})-\Delta  \hat{\omega}&=\Ra \partial_1 T \qquad \quad \quad \  \text{in}\ \  \Omega\,,\\ 
\hat\omega|_{t=0}&=|\omega_0|-\Lambda  \ \ \qquad \ \ \text{in}\ \  \Omega, \\
\omega&=0 \qquad\qquad \ \  \ \ \quad  \text{on}\ \  \{x_2=0\}\cup \{x_2=1\}.
\end{align*} 
We now perform $L^p$ estimates; multiplying by $\hat{\omega}|\hat{\omega}|^{p-2}$ where $p> 2$ we find 
\[
\frac{1}{p} \frac{\rmd}{\rmd t} \|\hat\omega\|_{L^p}^p + (p-1) \int_\Omega |\nabla\hat \omega|^2 |\hat\omega|^{p-2} \rmd x= -\Ra \int_\Omega \partial_1( \hat{\omega} |\hat\omega|^{p-2}) T \rmd x.
\]
We bound using Cauchy-Schwarz and Young's inequality
\begin{align*} \nonumber
\Ra \left|\int_\Omega \partial_1( \hat{\omega} |\hat\omega|^{p-2}) T \rmd x\right| &\leq (p-1) \left(\int_\Omega |\nabla\hat \omega|^2 |\hat\omega|^{p-2} \rmd x\right)^{\frac{1}{2}} \left(\Ra^2\int_\Omega |\hat\omega|^{p-2} T^2 \rmd x\right)^{\frac{1}{2}}\\
&\leq \frac{p-1}{2}\int_\Omega |\nabla\hat \omega|^2 |\hat\omega|^{p-2} \rmd x + \frac{p-1}{2}|\Omega|^{\frac{2}{p}}\Ra^2 \|\hat\omega\|_{L^p}^{p-2},
\end{align*} 
where we used that $\|T\|_{L^\infty}=1$. Thus we obtain
\[
\frac{1}{p} \frac{\rmd}{\rmd t} \|\hat\omega\|_{L^p}^p + \frac{p-1}{2} \int_\Omega |\nabla\hat \omega|^2 |\hat\omega|^{p-2} \rmd x\le \frac{p-1}{2}|\Omega|^{\frac{2}{p}}\Ra^2 \|\hat\omega\|_{L^p}^{p-2}.
\]
Finally, since $\hat\omega$ vanishes on the boundary, we have the Poincar\'e inequality
\[
 \int_\Omega |\nabla\hat \omega|^2 |\hat\omega|^{p-2} \rmd x=\frac{4}{p^2}\| \nabla |\omega|^{\frac{p}{2}}\|_{L^2}^2\ge \frac{4}{p^2C_p^2} \|  |\omega|^{\frac{p}{2}}\|_{L^2}^2= \frac{4}{p^2C_p^2} \| \omega\|_{L^p}^p
\]
Thus we obtain (dividing through by $ \|\hat\omega\|_{L^p}^{p-2}$) the inequality 
\[
\frac{\rmd}{\rmd t} \|\hat\omega\|_{L^p}^2\leq - \frac{p-1}{2} \frac{4}{p^2C_p^2} \|\hat\omega\|_{L^p}^2 + \ \frac{p-1}{2}|\Omega|^{\frac{2}{p}}\Ra^2 .
\]
It follows that for all $t\geq 0$
\be\label{vortbound}
 \|\hat\omega(t)\|_{L^p} \leq \|\hat\omega_0\|_{L^p}e^{-t\frac{(p-1)}{p^2C_p^2} } + \frac{pC_p}{2}|\Omega|^{\frac{1}{p}}\Ra\le C\Big(\|\omega_0\|_{L^p}e^{-\frac{t}{C}} +\Lambda +\Ra\Big),\quad C=C(p, \Gamma).
\ee
Given this bound, we estimate $\Lambda$ using interpolation as follows
\begin{align*}
\Lambda \leq \frac{1}{\LL} \|u\|_{L^\infty(\Omega\times (0, T))}& \leq  \frac{C}{\LL} \|u\|_{L^\infty([0, T]; L^2_x)}^\theta  \|\na u\|_{L^\infty([0, T]; L^p_x)}^{1-\theta}+ \frac{C}{\LL}\| u\|_{L^\infty([0, T]; L^2_x)}\\
&  \leq   \frac{C}{\LL} \|u\|_{L^\infty([0, T]; L^2_x)}^\theta  \|\omega\|_{L^\infty([0, T]; L^p_x)}^{1-\theta} + \frac{C}{\LL}\| u\|_{L^\infty([0, T]; L^2_x)}\\
&\leq  C_\ve \left(\frac{1}{\LL^{\frac{1}{\theta}}}+\frac{1}{\LL}\right)\|u\|_{L^\infty([0, T]; L^2_x)} +  \ve \|\omega\|_{L^\infty([0, T]; L^p_x)},
\end{align*}
where $\theta= \frac{p-2}{2p-2}\in (0, 1)$, $\ve>0$ is arbitrary and, appealing to Lemma \ref{ineqlem}, we used $\|\nabla u\|_{L^p}\lesssim \| \omega\|_{L^p}$.  By virtue of Lemma \ref{bndu}, for $\LL\ge 1$ we obtain
\[
\Lambda\le C_\ve\left(\frac{1}{\LL}\| u_0\|_{L^2}+\Ra\right)+ \ve \|\omega\|_{L^\infty([0, T]; L^p_x)}.
\]
 In view of this, \eqref{minmaxeqn} and \eqref{vortbound}, choosing $\varepsilon$ small enough gives
\[
\|\omega\|_{L^\infty([0, T]; L^p)} \le C\Big(\|\omega_0\|_{L^p} +\frac{1}{\LL}\| u_0\|_{L^2} +\Ra\Big),
\]
where $C$ is independent of $T$. Since $T>0$ is arbitrary, this completes the proof.
\end{proof}

An immediate consequence of the enstrophy balance \eqref{enstrophyeq} and the uniform vorticity bound \eqref{vorticity:ub} is the following global balance

\begin{cor}[Average Enstrophy Balance]\label{avenstbal}
We have the balance for long-time averages
\be\label{aventrophy}
\langle |\nabla \omega|^2\rangle  =  \frac{1}{\LL}\left( \langle p\partial_1 u_1 \big|_{x_2=1} \rangle + \langle p\partial_1 u_1 \big|_{x_2=0} \rangle \right)+\Ra \langle \omega \partial_1 T\rangle.
\ee
\end{cor}
\section{Proof of Theorem \ref{theorem}}

The theorem follows by an application of the background field method \cite{DC96}. This method is based on adopting the ansatz 
\be\label{tempdec}
T(x_1,x_2,t)=:\tau(x_2)+\theta(x_1,x_2,t).
\ee
We choose the ``background" profile $\tau:[0,1]\to [0,1]$ to be the continuous function given by
\begin{align}\label{taudef}
\tau(z) :=1-\frac{1}{2\delta}\begin{cases} z & z\in [0,\delta]\\
 \delta & z\in (\delta, 1-\delta)\\
z +2\delta-1 & z\in [1-\delta, 1]
\end{cases},
\end{align}
for some $\delta>0$ to be chosen later in the proof.
Note that
\begin{align}\label{dtaudef}
\tau'(z) =-\frac{1}{2\delta}\begin{cases} 1 & z\in [0,\delta)\\
 0 & z\in (\delta, 1-\delta)\\
1 & z\in (1-\delta, 1]
\end{cases}.
\end{align}
Note that $ \|\tau'\|_{L^2([0,1])}^2  = \frac{1}{2\delta}$. Note that $\theta$ vanishes at the boundaries $x_2=\{0,1\}$.

\begin{prop} With $\theta$ and $\tau$ defined by \eqref{tempdec} and \eqref{taudef}, the following identity holds
\begin{align}\label{eq:Nu-delta}
\Nu-  \frac{1}{2\delta} =-  \langle|\nabla \theta|^2\rangle   -2\langle \tau' u_2\theta\rangle.
\end{align}
\end{prop}
\begin{proof}
According to Proposition \ref{nussident}, the decomposition \eqref{tempdec} and the profile \eqref{dtaudef}, we have
\begin{align}
\Nu&= \langle|\nabla \theta|^2\rangle + \|\tau'\|_{L^2([0,1])}^2 + 2\langle \tau' \partial_2 \theta \rangle.
\end{align}
Inserting now the ansatz \eqref{tempdec} into \eqref{e3}, we find the fluctuation
$\theta$ satisfies
\begin{align} \label{thetaeqn}
\partial_t\theta+u_2\tau'+u\cdot\nabla \theta-\Delta\theta-\tau''&=0 \ \qquad \ \ \text{in}\ \  \Omega\,,\\  \label{te2}
\theta&=0 \qquad \ \  \text{on}\ \  \{x_2=0\}\cup \{x_2=1\}.
\end{align} 
Integrating \eqref{thetaeqn} against $\theta$ and taking the long-time average (using the fact that $\theta$, like $T$,  is uniformly bounded in time), we obtain
\begin{equation}\label{EE-theta}
 \langle \tau'\partial_2\theta\rangle  =-  \langle|\nabla \theta|^2\rangle -\langle \tau' u_2\theta\rangle.
\end{equation}
This argument can be made rigorous by smooth approximation of the profile $\tau$.
Inserting this equality above yields the claimed identity.
\end{proof}
Similarly to the bound of Doering--Constantin for the no-slip boundary condition \cite{DC96}, we have
\begin{lemma}\label{Lem-Nu12}
For any $\LL> 0$, we have $\Nu \lesssim \Ra^{\frac{1}{2}}$.
\end{lemma}
\begin{proof}
 Equation \eqref{eq:Nu-delta} implies $\Nu\leq  \frac{1}{2\delta}-2\langle \tau' u_2\theta\rangle$. Since $\tau'=\frac{1}{2\delta}$ on its support $(0, \delta)\cup (1, 1-\delta)$ and $\theta$ and $u_2$ vanish on $x_2=0,\ 1$, we have
 \[
 |\theta(x_1, x_2)|\le \sqrt{\delta}\|\p_2 \theta(x_1, \cdot)\|_{L^2(0, 1)}\quad\forall x_2\in (0, \delta)\cup (1, 1-\delta)
 \]
 and similarly  for $u_2$. Consequently,
 \[
 \frac{1}{\Gamma}\int_0^\Gamma\int_0^1 2 |\tau' u_2\theta| \rmd x_2 \rmd x_1\le \delta \frac{1}{\Gamma}\| \p_2 u_2\|_{L^2(\Omega)}\| \p_2\theta\|_{L^2(\Omega)}.
\]
Integrating in time and applying the Cauchy-Schwarz inequality gives
\be\label{est:utt:1}
|\langle -2 \tau' u_2\theta\rangle|
\le 2 \delta\langle |\partial_2u_2|^2\rangle^\frac12\langle |\partial_2\theta|^2\rangle^\frac12.
\ee
Appealing to Proposition \ref{nussident} and Corollary \ref{aveenbal} we deduce 
\begin{equation}
\Nu\leq\frac{1}{2\delta}+ 2\delta (\Nu)^{\frac 12}((\Nu-1)\Ra)^{\frac 12}\lesssim\frac{1}{2\delta}+2\delta \Nu\Ra^{\frac 12} \,.
\end{equation} 
Choosing $\delta\sim \Nu^{-\frac 12}\Ra^{-\frac 14}\,$ by balancing the contributions of each term
yields
$\Nu\lesssim \Ra^{\frac 12}\,$.
\end{proof}
To improve the bound, we follow \cite{WD11} by using the energy and enstrophy balances
\begin{align*}
\textbf{(a)}&:=  \langle |\nabla \omega|^2\rangle  -    \frac{1}{\LL}\left( \langle p\partial_1 u_1 \big|_{x_2=1} \rangle + \langle p\partial_1 u_1 \big|_{x_2=0} \rangle \right)-\Ra \langle \omega \partial_1 T\rangle,\\
\textbf{(b)}&:=  \langle |\nabla u|^2\rangle+ \frac{1}{\LL}\left( \langle u_1^2\big|_{x_2=1} \rangle + \langle u_1^2\big|_{x_2=0} \rangle \right)  -  \Ra( \Nu-1).
\end{align*}
Note that $\textbf{(a)}=\textbf{(b)}=0$ by Corollary \ref{aveenbal} and \ref{avenstbal}.  Thus in view of \eqref{eq:Nu-delta} we have 
\be
\Nu  = \frac{1}{2\delta}-  \langle|\nabla \theta|^2\rangle   -2\langle \tau' u_2\theta\rangle -\frac{b}{\Ra} \textbf{(b)}- a\textbf{(a)},
\ee
for all $b\in [0,1)$ and $a\in \mathbb{R}$. 
\begin{prop}\label{qdef}
Let $\delta>0$, $b\in [0,1)$, $a>0$ and $M>0$.  Then the following identity holds
\begin{align} \label{Nubdd}
(1-b)\Nu+b &= \frac{1}{2\delta}+  M \Ra^2 - \mathcal{Q}[\theta, u, \tau],
\end{align} 
where $\mathcal{Q}[\theta, u, \tau]$ is defined by
\begin{align} \nonumber
\mathcal{Q}[\theta, u, \tau]&:=  M \Ra^2+   \langle|\partial_1 \theta|^2\rangle+ \langle|\partial_2 \theta|^2\rangle    +2\langle \tau' u_2\theta\rangle \\ \nonumber
&\qquad+ \frac{b}{\Ra} \langle |\omega|^2\rangle  + \frac{b}{\Ra\LL}\left( \langle u_1^2\big|_{x_2=1} \rangle + \langle u_1^2\big|_{x_2=0} \rangle \right)\\
&\qquad \quad  + a \langle |\nabla \omega|^2\rangle  -  \frac{a}{\LL} \left( \langle p\partial_1 u_1 \big|_{x_2=1} \rangle + \langle p\partial_1 u_1 \big|_{x_2=0} \rangle \right)-a\Ra \langle \omega \partial_1 \theta \rangle.
\end{align}
\end{prop}
The strategy is to show that $\mathcal{Q}$ is non-negative for an appropriate choice of $\delta:=\delta(\Ra)$.  Then \eqref{Nubdd} will yield the desired bound on the Nusselt number. This requires bounds for the pressure and for $ 2\langle \tau' u_2\theta\rangle$, where the former is handled by virtue of \eqref{bound:p} and the latter requires a bound different from \eqref{est:utt:1}. The main result is
\begin{prop}
There exists a universal constant $L_0>0$ such that for all $\LL\ge L_0$  and $\Pr$ such that $\LL ^2\Pra^2\geq \Ra^{\frac{3}{2}}$, we have  
\be\label{mainbound:Nu}
\Nu \lesssim \Ra^{\frac{5}{12}}+\LL^{-2}\Ra^{\frac 12},\qquad\forall \\Ra>1.
\ee Here, the implicit constant depends only on $\Gamma$, $\| T_0\|_{L^\infty}$ and $\|u_0\|_{W^{1, r}}$ for any fixed $r\in (2, \infty)$.
\end{prop}
\begin{proof}
First we use Cauchy-Schwarz and Young's inequality to get
\be
|a\Ra \langle \omega \partial_1 \theta \rangle| \leq \frac{a^2\Ra^2}{2} \langle |\omega|^2\rangle  +  \frac{1}{2} \langle|\partial_1 \theta|^2\rangle,
\ee
so that $\mathcal{Q}$ of Proposition \ref{qdef} enjoys the lower bound
\be\label{Q:9}
\begin{aligned}
\mathcal{Q}[\theta, u, \tau]&\geq  M \Ra^2 +  \frac{1}{2}  \langle|\partial_1 \theta|^2\rangle + \langle|\partial_2 \theta|^2\rangle    +2\langle \tau' u_2\theta\rangle+ \left( \frac{b}{\Ra} - \frac{a^2 \Ra^2}{2}  \right) \langle |\omega|^2\rangle  + a \langle |\nabla \omega|^2\rangle \\
&\qquad + \frac{b}{\Ra\LL}\left( \langle u_1^2\big|_{x_2=1} \rangle + \langle u_1^2\big|_{x_2=0} \rangle \right)  -  \frac{a}{\LL}\left( \langle p\partial_1 u_1 \big|_{x_2=1} \rangle + \langle p\partial_1 u_1 \big|_{x_2=0} \rangle \right).
\end{aligned}
\ee
Note that from the Sobolev trace inequality and the incompressibility, we have
\begin{align*} \nonumber
\frac{a}{\LL}\left| \langle p\partial_1 u_1 \big|_{x_2=1} \rangle + \langle p\partial_1 u_1 \big|_{x_2=0} \rangle \right| \leq  \frac{C_1a}{\LL} \langle \| p\|_{H^1}\|\partial_2 u_2\|_{L^2}\rangle \leq   \frac{C_1a}{\LL} \langle \|\nabla p\|_{L^2}\|\partial_1 \omega\|_{L^2}\rangle,
\end{align*}
 where we used  \eqref{elliptic:u2:2} and $C_1=C_1(\Gamma)$.
To bound the pressure, we recall from \eqref{bound:p} that for any $r\in (2, \infty)$, 
\[
\| p\|_{H^1(\Omega)} \le C\Big(\frac{1}{\LL} \|\partial_1 \omega\|_{L^2(\Omega)}+\Ra \| T\|_{L^2(\Omega)}+ \frac{1}{\Pr}\| \omega\|_{L^2(\Omega)}\|\omega\|_{L^r(\Omega)}\Big). 
\]
Recall also from Lemma \ref{unifenstbdd} that $\| \omega\|_{L^r}\le C(\| u_0\|_{W^{1, r}}+\Ra)$ and hence
\[
C_1\| p\|_{H^1(\Omega)} \le C_2\Big(\frac{1}{\LL} \|\partial_1 \omega\|_{L^2(\Omega)}+\Ra + \frac{\| u_0\|_{W^{1, r}}+\Ra}{\Pr}\| \omega\|_{L^2(\Omega)}\Big). 
\]
Using Young's inequality  yields
\begin{align*}
\frac{aC_1}{\LL} \|\nabla p\|_{L^2}\|\partial_1 \omega\|_{L^2}&\le \frac{aC_2}{\LL^2}\|\partial_1 \omega\|_{L^2}^2+\frac{aC_2}{\LL}\|\partial_1 \omega\|_{L^2}\Big(\Ra + \frac{\| u_0\|_{W^{1, r}}+\Ra}{\Pr} \|\omega\|_{L^2}\Big)\\
& \le  \frac{aC_2}{\LL^2}\|\partial_1 \omega\|^2_{L^2}+\frac{a}{2}\|\partial_1 \omega\|^2_{L^2}+\frac{aC_2^2}{2\LL^2}\left(\Ra^2+ \frac{\| u_0\|^2_{W^{1, r}}}{\Pr^2} \|\omega\|^2_{L^2}+\frac{\Ra^2}{\Pr^2} \|\omega\|_{L^2}^2 \right).
\end{align*}
Choosing $M=\frac{aC_2^2}{2\LL^2}$ in the definition on $\mathcal{Q}$, we find
\begin{align} \nonumber
\mathcal{Q}[\theta, u, \tau]&\geq \frac{1}{2}  \langle|\partial_1 \theta|^2\rangle + \langle|\partial_2 \theta|^2\rangle    +2\langle \tau' u_2\theta\rangle+  \left( \frac{b}{\Ra} - \frac{a^2\Ra^2}{2}  -\frac{aC_2^2\| u_0\|^2_{W^{1, r}}}{2\LL^2\Pr^2}-  \frac{aC_2^2\Ra^2}{2\LL^2\Pr^2}\right) \langle |\omega|^2\rangle\\ \label{Q:10}
&\qquad \qquad+ a \Big(\frac12- \frac{C_2}{\LL^2}\Big) \langle |\nabla \omega|^2\rangle.
\end{align}
\begin{lemma}\label{est:u2tt:2}
For some $C_0>0$ and any $\ve>0$ we have 
\begin{enumerate}
\item[(a)]
\be
|2\langle \tau' u_2\theta\rangle|\le  \frac12\langle|\partial_2 \theta|^2\rangle  +C_0\delta^6\ve^{-1}\langle|\omega|^2\rangle+ \frac{\ve}{4}\langle|\partial_1 \omega|^2\rangle,
\ee
\item[(b)]
\be
|2\langle \tau' u_2\theta\rangle|\le  \frac12\langle|\partial_2 \theta|^2\rangle  +C_0\delta^4\ve^{-\frac23}\langle|\omega|^2\rangle+ \frac{\ve^2}{4}\langle|\partial_1^2 \omega|^2\rangle.
\ee
\end{enumerate}
\end{lemma}
\begin{proof}[Proof of Lemma \ref{est:u2tt:2}]
Note that 
\[
2\int_0^1\tau' u_2\theta\rmd x_2 = \frac{1}{\delta}\left( \int_0^\delta u_2\theta\rmd x_2  + \int_{1-\delta}^1u_2\theta \rmd x_2  \right).
\]
We shall consider the first integral; the second one is treated similarly. Since $\theta$ and $u_2$ vanish on $x_2=0$, we have
\[
| \theta(x_1, x_2)|\le \sqrt{x_2}\| \p_2\theta(x_1, \cdot)\|_{L^2(0, x_2)},\quad |u_2(x_1, x_2)|\le x_2\| \p_2u_2(x_1, \cdot)\|_{L^\infty(0, 1)}\quad\forall x_2\in (0, 1),
\]
where, for the second bound, we used the fundamental theorem of calculus to have $u_2(x_1,x_2)=\int_0^{x_2}\partial_{2}u_2(x_1,z)\, dz\leq x_2\sup_{0\leq z\leq x_2}|\partial_2u_2(x_1,\cdot)|$.
Noting that  $\int_0^1\partial_2u_2(x_1, x_2)dx_2=0$, we deduce $\partial_2 u_2(x_1, z_0)=0$ for some $z_0=z_0(x_1)\in (0, 1)$.  Then by the fundamental theorem of calculus and H\"older's inequality, we obtain 
\be\label{interpolation:u2}
 |\partial_2u_2(x_1, x_2)|^2=2\left|\int_{z_0}^{x_2} \partial_2 u_2(x_1, z)\partial^2_2 u_2(x_1, z)\rmd z\right|\lesssim \| \p_2u_2(x_1, \cdot)\|_{L^2(0, 1)}\| \p_2^2u_2(x_1, \cdot)\|_{L^2(0, 1)}.
\ee
 Applying H\"older's inequality for $x_1$ yields 
\begin{align*}
I:=\frac{1}{\delta}\frac{1}{\Gamma}\left|\int_0^\Gamma \int_0^\delta {u}_2{\theta} \rmd x_2 \rmd x_1\right|&\lesssim \delta^\frac32\frac{1}{\Gamma}\| \p_2\theta\|_{L^2(\Omega)}\| \p_2u_2\|_{L^2(\Omega)}^\frac12 \|\p_2^2u_2\|_{L^2(\Omega)}^\frac12\\
&\le \frac{C}{\Gamma} \delta^\frac32\| \p_2\theta\|_{L^2(\Omega)}\| \omega\|_{L^2(\Omega)}^\frac12 \|\p_1\omega\|_{L^2(\Omega)}^\frac12,
\end{align*}
where we have used Lemma \ref{vortidents} and \eqref{elliptic:u2:2}.
\\
Proof of (a): From the above we have
\begin{align*}
I
&\leq \frac{C}{\Gamma}\| \p_2\theta\|_{L^2(\Omega)}\{\delta^\frac32 \ve^{-\frac14} \| \omega\|_{L^2(\Omega)}^\frac12\}\{\ve^{\frac14}\|\p_1\omega\|_{L^2(\Omega)}^\frac12\}.
\end{align*}
 Taking the time average and using the H\"older and Young inequalities, we deduce 
\[
\langle I\rangle \le  \frac14\langle|\partial_2 \theta|^2\rangle  +C\delta^6\ve^{-1}\langle|\omega|^2\rangle+ \frac{\ve}{8}\langle|\partial_1 \omega|^2\rangle.
\]
Proof of (b): As in \eqref{interpolation:u2}, we have the interpolation inequality 
$
\|\partial_1\omega\|_{L^2(\Omega)}^2
\leq \|\omega\|_{L^2(\Omega)} \|\partial_1^2\omega\|_{L^2(\Omega)}.
$
Thus we obtain the bound
\begin{align*}
I
&\leq \frac{C}{\Gamma}\| \p_2\theta\|_{L^2(\Omega)}\{\delta^\frac32 \ve^{-\frac14} \| \omega\|_{L^2(\Omega)}^\frac34\}\{\ve^{\frac14}\|\p_1^2\omega\|_{L^2(\Omega)}^\frac14\}\\
&\leq \frac{1}{4} \| \p_2\theta\|_{L^2(\Omega)}^2 +  C_0\{\delta^\frac32 \ve^{-\frac14} \| \omega\|_{L^2(\Omega)}^\frac34\}^{\frac 83}  +\frac{1}{8} \{\ve^{\frac14}\|\p_1^2\omega\|_{L^2(\Omega)}^\frac14\}^{8}.
\end{align*}
The proof is complete.
\end{proof}

Applying Lemma \ref{est:u2tt:2} (a) with $\ve=a$ to \eqref{Q:10}, we find
\begin{align} \nonumber
\mathcal{Q}[\theta, u, \tau]&\geq \frac{1}{2}  \langle|\partial_1 \theta|^2\rangle + \frac12\langle|\partial_2 \theta|^2\rangle +  \left( \frac{b}{\Ra}- \frac{a^2 \Ra^2}{2}  -\frac{aC_2^2\| u_0\|^2_{W^{1, r}}}{2\LL^2\Pr^2}-  \frac{aC_2^2\Ra^2}{2\LL^2\Pr^2}-C_0\delta^6a^{-1}\right) \langle |\omega|^2\rangle\\ \label{Q:11}
&\qquad \qquad+ a \Big(\frac14-  \frac{C_2}{\LL^2} \Big) \langle |\nabla \omega|^2\rangle.
\end{align}
Clearly, the coefficient of $\langle |\nabla \omega|^2\rangle$ in \eqref{Q:11} is positive for sufficiently large $\LL$. Fixing an arbitrary $b\in (0, 1)$ and imposing $\LL^2 \Pr^2\ge \Ra^\frac32$ and $a=a_0\Ra^{-\frac32}$ gives 
\[
 A:=\frac{b}{\Ra}- \frac{a^2 \Ra^2}{2}  -\frac{aC_2^2\| u_0\|^2_{W^{1, r}}}{2\LL^2\Pr^2}-  \frac{aC_2^2\Ra^2}{2\LL^2\Pr^2}\ge  \frac{b}{\Ra}- \frac{a_0^2}{2\Ra}  -\frac{a_0C_2^2\| u_0\|^2_{W^{1, r}}}{\Ra^3}-  \frac{a_0C_2^2}{2\Ra}.
 \]
We choose 
\[
a_0=\frac{b}{100C_2^2}\min\big\{1, \frac{\Ra^2}{\| u_0\|^2_{W^{1, r}}}\big\}
\]
so that $A\ge \frac{b}{2\Ra}$. Letting  $\delta$ solve  $\frac{b}{2\Ra}=2C_0\delta^6a_0^{-1}\Ra^\frac32$, the coefficient of $\langle |\omega|^2\rangle$ in \eqref{Q:11} is positive and hence $\mathcal{Q}$ is positive. This gives 
 \[
 \delta=\left(\frac{a_0b}{4C_0}\right)^{\frac16}\Ra^{-\frac{5}{12}}.
 \]
In view of \eqref{Nubdd} with $M=\frac{aC_2^2}{2\LL^2}$, we  obtain 
$
\Nu \le \frac12\big(\frac{4C_0}{a_0b}\big)^{\frac16}\Ra^{\frac{5}{12}}+ \frac{a_0C_2^2}{2}\LL^{-2}\Ra^{\frac12}$. Inserting $a_0$ we finally arrive at \eqref{mainbound:Nu}.
\end{proof}
For $\LL \in (0, L_0)$, we have $\Nu\lesssim \Ra^{\frac12}$ according to Lemma \ref{Lem-Nu12}, and hence the bound \eqref{mainbound:Nu} is still valid. If $\LL=\infty$, the entire argument follows the same way in view of Remark \ref{rema:Poincare}.

\begin{rem}[A proof of the $\Pr=\infty$ result of Whitehead]\label{whiteheadproof}
 If $\Pr=\infty$, the inertial term in the momentum equation vanishes. 
We work in $2d$ for the sake of simplicity. The key observation of Whitehead is that from \eqref{ve1} with $\Pr=\infty$ we have
\be
\langle |\partial_1 \theta|^2 \rangle = \frac{1}{\Ra^2} \langle |\Delta \omega|^2 \rangle \geq \frac{1}{C}  \langle |\partial_1^2 \omega|^2\rangle,
\ee
since $\partial_1 \theta=\partial_1 T$ and according to Lemma \ref{ineqlem2}, we have $ \langle |\partial_1^2 \omega|^2 \rangle \leq  C\langle |\Delta \omega|^2 \rangle $ for some $C>0$ for any $\LL>0$.
Applying Lemma \ref{est:u2tt:2} (b)  to \eqref{Q:10} with $M=a=0$, we find
\begin{align} \nonumber
\mathcal{Q}[\theta, u, \tau]&\geq  \left( \frac{b}{\Ra}-C_0\delta^4\ve^{-\frac{2}{3}}\right) \langle |\omega|^2\rangle+ \Big(\frac{1}{2C\Ra^2}- \frac{\ve^2}{8} \Big) \langle |\partial_1^2 \omega|^2\rangle.
\end{align}
The bound $\mathcal{Q}[\theta, u, \tau]\geq 0$ follows by choosing $\ve= C^{-1/2}\Ra^{-1}$ and $\delta \sim \Ra^{-5/12}$.
\end{rem}

\appendix

\section{Some elliptic estimates}
Here we record some useful identities/inequalities involving the vorticity. 
\begin{lemma}\label{vortidents}
With $\omega = \nabla^\perp \cdot u$,  the following identities hold
\begin{itemize}
\item
$\|\nabla u\|_{L^2}= \|\omega\|_{L^2}$,
\item $\|\Delta u\|_{L^2}= \|\nabla\omega\|_{L^2}$.
\end{itemize}
\end{lemma}
\begin{proof}
The second identity is a consequence of $\Delta u=\nabla^\perp \omega$. Next we prove the first identity. By the periodicity in $x_1$ and the boundary condition $u_2=0$ on $\{x_2=0\}\cup \{x_2=1\}$, we have  
\begin{align*}
\sum_{i, j=1, 2}\int_\Omega \p_j u_i\p_j u_i \rmd x&=-\int_\Omega u\cdot \Delta u \rmd x+\int_0^\Gamma u_1\p_2u_1\Big|_{x_2=0}^{x_2=1}\rmd x_1\\
&=-\int_\Omega u\cdot \nabla^\perp \omega \rmd x+\int_0^\Gamma u_1\p_2u_1\Big|_{x_2=0}^{x_2=1}\rmd x_1\\
&=\int_\Omega |\omega|^2 \rmd x+\int_0^\Gamma u_1(\p_2u_1+\omega)\Big|_{x_2=0}^{x_2=1}\rmd x_1=\int_\Omega |\omega|^2 \rmd x,
\end{align*}
where we have used that $\p_2u_1+\omega=\p_1u_2=0$ on $\p\Omega$. 
\end{proof}
\begin{lemma}\label{ineqlem}
For any $m\ge 1$ and $p\in (1, \infty)$, there exists $C$ such that $\|\nabla u\|_{W^{m,p}}\leq  C\|\omega\|_{W^{m,p}}$.
\end{lemma}

\begin{proof}
Let $\psi$ be the streamfunction for $u$, i.e. $u=\nabla^\perp \psi$ such that 
\begin{align*}
  \Delta \psi  &=  \omega \qquad\quad \ \ \text{in}\ \  \Omega\,,\\  
\qquad \psi &= 0 \qquad \ \ \  \ \ \text{on}\ \   \{x_2=0\},\\
\qquad \psi &= c(t) \qquad \ \  \text{on}\ \   \{x_2=1\},
\end{align*}
for some possibly time dependent but spatially constant $c(t)$. Consequently, $\p_1\psi$ satisfies 
\begin{align} \label{es1}
  \Delta \p_1\psi  &=  \p_1\omega \qquad\quad \ \ \text{in}\ \  \Omega\,,\\   \label{es2}
\qquad \p_1\psi &= 0 \qquad \ \ \  \ \ \text{on}\ \   \{x_2=0\}\cup \{x_2=1\}.
\end{align}
Fix $k\ge 1$ and $p\in (1, \infty)$. By elliptic regularity, we have
\begin{align}\label{elliptic:u2:1}
&\|\nabla u_2\|_{L^p}=\|\nabla\partial_1\psi \|_{L^p}\le C\| \omega\|_{L^p},\\ \label{elliptic:u2:2}
&\|u_2\|_{W^{1+k, p}}=\|\partial_1\psi \|_{W^{1+k, p}} \leq C \|\partial_1\omega\|_{W^{k-1, p}}.
\end{align}
Now note that by divergence-free and the definition of the vorticity we have $\partial_1 u_1 = - \partial_2 u_2$ and $\partial_2 u_1  = \partial_1 u_2 - \omega$. Therefore, for any $m\ge 0$, we have the bound
\begin{align*}
&\|\nabla u_1\|_{W^{m, p}} \leq C\left( \|\nabla u_2\|_{W^{m, p}}  +  \|\omega\|_{W^{m, p}} \right) \leq C \|\omega\|_{W^{m, p}}.
\end{align*}
\end{proof}

\begin{lemma}\label{ineqlem2}
With $\omega= \nabla^\perp \cdot u$, we have $\|\partial_1 \omega\|_{L^2} \leq C\|\Delta \omega\|_{L^2}$ for some $C>0$.  
\end{lemma}

\begin{proof}
From \eqref{es1}--\eqref{es2} we have $ \Delta \partial_1 u_2 =  \p_1^2\omega$ in $\Omega$ and  $\partial_1 u_2= 0$ on  $\{x_2=0\}\cup \{x_2=1\}$ since $\partial_1$ is a tangential derivative. It follows
\[
 \int_\Omega \Delta^2\partial_1 u_2 \partial_1 u_2\rmd x_1 \rmd x_2  =   \int_\Omega  \Delta \p_1^2\omega\partial_1 u_2\rmd x_1 \rmd x_2.
\]
First note 
\begin{align*}
 \int_\Omega \Delta^2\partial_1 u_2 \partial_1 u_2\rmd x_1 \rmd x_2  &= - \int_\Omega \nabla \Delta \partial_1 u_2 \cdot \nabla \partial_1 u_2 \rmd x_1 \rmd x_2 \\
    &= \|\Delta \partial_1 u_2\|_{L^2(\Omega)}^2 -   \int_0^\Gamma \partial_2^2 \partial_1 u_2 \partial_2 \partial_1 u_2 \rmd x_1 \rmd x_2 \Big|_{x_2=0}^1\\
        &= \|\Delta \partial_1 u_2\|_{L^2(\Omega)}^2 -   \int_0^\Gamma   \partial_1^2 \partial_2 u_1  \partial_1^2  u_1\rmd x_1 \rmd x_2 \Big|_{x_2=0}^1\\
                &= \|\Delta \partial_1 u_2\|_{L^2(\Omega)}^2 +\frac{1}{\LL}    \int_0^\Gamma  (  \partial_1^2  u_1)^2 \rmd x_1 \rmd x_2 \Big|_{x_2=1} +\frac{1}{\LL}     \int_0^\Gamma  (  \partial_1^2  u_1)^2 \rmd x_1 \rmd x_2 \Big|_{x_2=0}\\
                &\geq \|\Delta \partial_1 u_2\|_{L^2(\Omega)}^2,
\end{align*}
where we used incompressibility, the fact that $\partial_1^3 u_2$ is zero on the boundary and the boundary conditions \eqref{e4}--\eqref{e5}.  On the other hand
\begin{align*}
 \int_\Omega  \Delta \p_1^2\omega\partial_1 u_2\rmd x_1 \rmd x_2=  \int_\Omega  \Delta  \omega\partial_1^3u_2\rmd x_1 \rmd x_2 \leq \|\Delta \omega \|_{L^2(\Omega)} \|\partial_1^3 u_2 \|_{L^2(\Omega)}  \leq C\|\Delta \omega \|_{L^2(\Omega)} \|\Delta \partial_1 u_2 \|_{L^2(\Omega)}\,,
\end{align*}
where we used that, since $ \partial_1 u_2=0$ on the boundary, elliptic regularity tells us  $\|\partial_1^3 u_2 \|_{L^2(\Omega)} \leq \|\partial_1 u_2 \|_{H^2(\Omega)} \leq C \|\Delta \partial_1 u_2 \|_{L^2(\Omega)}$. Finally since 
$ \Delta \partial_1 u_2 \partial_1^2 \omega$, we are done.
\end{proof}

 \subsection*{Acknowledgments}   
We would like to remember and thank Charlie for his advice and encouragement, as well as for sharing his vision of science with us.  We thank J. Whitehead for insightful remarks and for letting us know about his unpublished result in the infinite Prandlt number case. We also thank  D. Goluskin and V. Martinez  for  useful discussions, and gratefully acknowledge Johannes L\"{u}lff for allowing us to use his simulation data to produce Figure \ref{fig:RB} (see \cite{L15} for simulation details).
Research of TD was partially supported by  NSF grant DMS-2106233.  HQN was partially supported by NSF grant DMS-19077. Research of CN was partially supported by the DFG-GrK2583 and DFG-TRR181.

\end{document}